\documentclass[a4paper,12pt]{amsart}

\usepackage[includehead,twoside=False]{geometry}

\usepackage[active]{srcltx}\usepackage{amsmath,amsthm,amsxtra}
\usepackage{amssymb}
\usepackage{mathabx}

\usepackage[mathscr]{eucal}

\usepackage{bm}
\usepackage[all]{xy}

\usepackage{aliascnt}

\theoremstyle{plain}
\newtheorem{thm}{Theorem}[section]
\newtheorem*{thm*}{Theorem}

\newaliascnt{prop}{thm}
\newaliascnt{cor}{thm}
\newaliascnt{lem}{thm}
\newaliascnt{claim}{thm}
\newaliascnt{defn}{thm}
\newaliascnt{ques}{thm}
\newaliascnt{conj}{thm}
\newaliascnt{fact}{thm}
\newaliascnt{rem}{thm}
\newaliascnt{ex}{thm}
\newtheorem{prop}[prop]{Proposition}
\newtheorem{cor}[cor]{Corollary}
\newtheorem{lem}[lem]{Lemma}

\newtheorem*{prop*}{Proposition}
\newtheorem*{cor*}{Corollary}
\newtheorem*{lem*}{Lemma}
\newtheorem*{claim*}{Claim}
\theoremstyle{definition}
\newtheorem{defn}[defn]{Definition}

\newtheorem*{defn*}{Definition}
\newtheorem*{ques*}{Question}
\newtheorem*{conj*}{Conjecture}

\newtheorem*{prob*}{Problem}

\newtheorem{rem}[rem]{Remark}
\newtheorem{ex}[ex]{Example}
\newtheorem*{fact*}{Fact}
\newtheorem*{rem*}{Remark}
\newtheorem*{ex*}{Example}
\aliascntresetthe{prop}
\aliascntresetthe{cor}
\aliascntresetthe{lem}
\aliascntresetthe{claim}
\aliascntresetthe{defn}
\aliascntresetthe{ques}
\aliascntresetthe{conj}
\aliascntresetthe{fact}
\aliascntresetthe{rem}
\aliascntresetthe{ex}

\usepackage[pointedenum]{paralist}
\usepackage{varioref}
\labelformat{equation}{\textnormal{(#1)}}
\labelformat{enumi}{\textnormal{#1}}
\def\equationautorefname~#1\null{#1\null}
\def\itemautorefname~#1\null{(#1)\null}

\usepackage[a4paper,dvipdfmx]{hyperref}
\def\textsectionN~{\textsection{}}

\newcommand{\sU}{\mathcal{U}}

\newcommand{\Xo}{X\spcirc}

\newcommand{\sQ}{\mathscr{Q}}
\newcommand{\sS}{\mathscr{S}}
\newcommand{\sHom}{\mathscr{H}om}

\newcommand{\Xdo}{{{X'}\spcirc}}
\newcommand{\Zo}{{Z\spcirc}}

\renewcommand\phi{\varphi}
\renewcommand\epsilon{\varepsilon}
\renewcommand\leq{\leqslant}
\renewcommand\geq{\geqslant}

\makeatletter
\newcommand{\set}{  \@ifstar{\@setstar}{\@set}}\newcommand{\@setstar}[2]{\{\, #1 \mid #2 \,\}}
\newcommand{\@set}[1]{\{ #1 \}}
\makeatother
\newcommand{\lin}[1]{\langle #1 \rangle}
\newcommand{\trans}[1][1]{\raisebox{#1ex}{\scriptsize\kern0.1em$t$\kern-0.1em}}

\newcommand{\PP}{\mathbb{P}}
\newcommand{\cP}{{\PP_{\!\! *}}}
\newcommand{\TT}{\mathbb{T}}
\newcommand{\PN}{\PP^N}

\newcommand{\GN}{\Gr(n, \PP^N)}

\newcommand{\Pv}[1][N]{(\PP^{#1})\spcheck}

\newcommand{\sO}{\mathscr{O}}
\newcommand{\ZZ}{\mathbb{Z}}

\newcommand{\Gr}{\mathbb{G}}
\newcommand\spcirc{^\circ}

\newcommand{\FF}{\mathbb{F}}

\DeclareMathOperator{\Pic}{Pic}

\DeclareMathOperator{\Sing}{Sing}

\DeclareMathOperator{\rank}{rk}

\DeclareMathOperator{\pr}{pr}

\def\Q{\mathbb{Q}}
\def\R{\mathbb{R}}
\def\C{\mathbb{C}}

\def\F{\mathbb{F}}
\def\r+{\mathbb{R}_{\geq 0}}

\def\ep{\varepsilon}

\def\r+{{\R}_{\geq 0}}
\def\q+{{\Q}_{\geq 0}}
\def\P{\mathbb{P}}
\def\arw{\rightarrow}
\def\*c{\C^{\times}}

\newcommand{\pmx}{P_m^{X}}
\newcommand{\pmxo}{{\pmx}\spcirc}

\def\C{\mathbb {C}}

\def\F{\mathbb {F}}
\def\G{\mathbb {G}}

\def\Q{\mathbb {Q}}
\def\R{\mathbb {R}}

\def\T{\mathbb {T}}

\newcommand{\calu}{\mathcal {U}}
\newcommand{\calv}{\mathcal {V}}

\makeatletter

\@addtoreset{equation}{section}
\makeatother

\newcommand\Gn{\Gr(n,\PN)}
\newcommand\Gm{\Gr(m,\PN)}

\newcommand\Lmx[1][x]{\Lambda_{#1}}

\newcommand{\Wm}[1][L]{\sigma_{m,x}^{#1}}

\newcommand{\snxM}[1][2]{\sigma_{n+#1,x}^M}

\newcommand{\Fnm}{\FF(n,m;\PN)}
\newcommand{\PmPm}{P_{n+1}^X \times_{X_{n+1}^*} P_{n+1}^X}
\newcommand{\PmPmo}{(P_{n+1}^X)\spcirc \times_{X_{n+1}^*} (P_{n+1}^X)\spcirc}

\title{On separable higher Gauss maps}
\author[K.~Furukawa]{Katsuhisa~FURUKAWA}
\address{Graduate School of Mathematical Sciences, 
  the University of Tokyo, Tokyo, Japan}
\email{katu@ms.u-tokyo.ac.jp}

\author[A.~Ito]{Atsushi~Ito}
\address{Department of Mathematics, Kyoto University,
  Kyoto, Japan}
\email{aito@math.kyoto-u.ac.jp}

\subjclass[2010]{14N05}
\keywords{higher Gauss map, linearity of a general contact locus, dual defect}

\begin{document}

\maketitle

\begin{abstract}
  We study
  the $m$-th Gauss map in the sense of F.~L.~Zak
  of a projective variety $X \subset \P^N$ over an algebraically closed field in any characteristic.
  For all integer $m$ with $n:=\dim(X) \leq m < N$,
  we show that the contact locus on $X$ of a general tangent $m$-plane
  is a linear variety if the $m$-th Gauss map
  is separable.
  We also show that for smooth $X$ with $n < N-2$,
  the $(n+1)$-th Gauss map 
  is birational if it is separable,
  unless $X$ is the Segre embedding $\PP^1 \times \PP^n \subset \PP^{2n-1}$.
  This is related to L. Ein's classification of varieties with small dual varieties in characteristic zero.
\end{abstract}

\section{Introduction}

Let $X \subset \PN$ be an $n$-dimensional non-degenerate projective variety
over an algebraically closed field in characteristic
$p \geq 0$, and let $m$ be an integer with $n \leq m < N$.
We define the \emph{$m$-th Gauss map} $\gamma_m$,
\[
\gamma_m: P_m^X := \overline{\set*{(L, x) \in \Gm \times X^{sm}}{\TT_xX \subset L}}\rightarrow \Gm,
\]
to be the first projection from the incidence $P_m^X$ to the Grassmann variety $\Gm$,
where $X^{sm}$ is the smooth locus of $X$,
and $\TT_xX \subset \PN$ is the embedded tangent
space to $X$ at $x$.
We set $X_m^* = \gamma_m(P_m^X) \subset \Gm$,
the set of $m$-planes tangent to $X$.
See \cite[I, \textsection{}2]{Zak}.

Note that $\gamma_n$ is identified with the (ordinary) Gauss map
$X \dashrightarrow \Gn: x \mapsto \TT_xX$,
and $X_{N-1}^*$ is the projective dual variety $X^* \subset \Pv$ of $X$.
For a tangent $m$-plane $L \in X_m^*$,
we consider the \emph{contact locus} on $X$ of $L$,
\begin{equation}\label{eq:conloc}
  \overline{\set*{x \in X^{sm}}{\TT_xX \subset L}},
\end{equation}
which is equal to
$\pi(\gamma_m^{-1}(L)_{red}) \subset X$ if $L \in X_m^*$ is general,
where
$\pi: P_m^X \rightarrow X$ is the second projection.
It is also called an \emph{$m$-th contact locus},
and is said to be \emph{general} if so is $L$.
\\

In characteristic $p=0$,
it was well know that a general $m$-th contact locus is a linear variety of $\PN$
for all $m$ with $n \leq m < N$;
the case of $m = n$ is the linearity of a general fiber of the Gauss map (P.~Griffiths and J.~Harris \cite[(2.10)]{GH}, F.~L.~Zak \cite[I, 2.3.~Theorem~(c)]{Zak})
and the case of any $m$ can be shown due to the reflexivity (S.~L.~Kleiman and R.~Piene \cite[pp.~108--109]{KP}).
H.~Kaji \cite{Kaji15} recently gave a formula for the degree of $X_m^*$ in the case when $\gamma_m$ is generically finite, in particular, the case when $X = v_d(\P^n) \subset \PN$ is the Veronese variety with $N = \binom{n+d}{d}$ as a generalization of Boole's formula for $\deg(X^*)$.

In characteristic $p > 0$,
the linearity of a general contact locus does not hold in general;
in particular for $m=n$, several authors gave examples of $X$
such that a general fiber of $\gamma_n$ is not a linear variety
(\cite[\textsection 7]{Wallace}, \cite[I-3]{Kleiman1986},
\cite[Example~4.1]{Kaji1986}, \cite{Kaji1989}, \cite[Example~2.13]{Rathmann}, \cite{Noma2001},
\cite[\textsection{}7]{Fukasawa2005}, \cite{Fukasawa2006}, \cite[\textsection{}5]{gmaptoric}, \cite[Theorem~1.3]{FI}).

However, it was known that a general $m$-th contact locus is a linear variety of $\PN$
if $\gamma_m$ is \emph{separable}
in the case when $m = N-1$ (by the Monge-Segre-Wallace criterion \cite[(2.4)]{HK}, \cite[I-1(4)]{Kleiman1986}),
and in the case when $m = n$ (the first author \cite{expshr}).

In this paper, we prove the following theorem for all $m$.

\begin{thm}\label{thm:m-th-linearity}
  Let $X \subset \PN$ be a projective variety
  in characteristic $p \geq 0$,
  and let $m$ be an integer with $n = \dim(X) \leq m < N$.
  Assume that the $m$-th Gauss map $\gamma_m$ is separable.
  Then a general $m$-th contact locus \autoref{eq:conloc} is a linear variety of $\PN$.
\end{thm}

If $\gamma_m$ is separable, then
so is $\gamma_{m'}$ for all $m'$ with $n \leq m' \leq m$;
this can be shown in the same way for the reflexive case.
See \autoref{thm:m-1-th-sep}.
On the other hand, for $m'' > m$,
$\gamma_{m''}$ can be inseparable even if $\gamma_m$ is separable;
H.~Kaji \cite{Kaji2003} and S.~Fukasawa \cite{Fukasawa2006-3} \cite{Fukasawa2007}
gave such examples for the case $m'' = N-1$ and $m = n$.
\\

F.~L.~Zak showed that
a contact locus of any $L \in \gamma_m(\pi^{-1}(X^{sm}))$ is of dimension
$\leq m-n+\dim(\Sing(X))+1$
\cite[I. \textsection{}2, Theorem 2.3(a)]{Zak}.
Hence, if $X$ is smooth, then
\emph{any} $m$-th contact locus is of dimension $\leq m-n$.
Considering the case of $m = n$
and combining with the linearity (in particular, irreducibility)
of a general fiber of the $n$-th Gauss map $\gamma_n$,
it holds that $\gamma_n$ is in fact birational if it is separable
(e.g., the characteristic $p = 0$).

In the case $m=n+1$,
any $(n+1)$-th contact locus is of dimension $\leq 1$
if $X$ is smooth.
We study when the equality holds,
and obtain the following theorem.

\begin{thm} \label{thm:gamma_n+1-sep-bir-unlessP1Pn-1}
  Let $X \subset \PN$ be a smooth non-degenerate projective variety in characteristic $p \geq 0$
  such that $n = \dim(X) < N-2$ and $\gamma_{n+1}$ is separable.
  Then $\gamma_{n+1}$ is in fact birational, unless
  $X$ is the Segre embedding $\P^1 \times \P^{n-1} \subset \PP^{2n-1}$.
\end{thm}

In characteristic $p = 0$,
since the case of $n = N-2$ follows from \cite[Theorem~3.4]{Ein},
the statement holds for all $n \leq N-2$,
and we have that $\P^1 \times \P^{n-1} \subset \PP^{2n-1}$
is only the smooth projective variety which satisfies
the equality for Zak's inequality on dimensions of $m$-th contact loci
for any $n \leq m < N$.
See \autoref{thm:delta_m=m-n} and \autoref{thm:delta_m=m-n_rem} for detail.

Note that if $\gamma_{n+1}$ is separable and generically finite,
then it is birational
because of
\autoref{thm:m-th-linearity}.
Hence the main point of the proof of \autoref{thm:gamma_n+1-sep-bir-unlessP1Pn-1}
is to study the case when $\gamma_{n+1}$ is not generically finite.
\\

A background of our study is the classification of varieties with small dual varieties by L.~Ein \cite{Ein};
in characteristic $p=0$, he showed that if a smooth projective variety $X \subset \PN$ satisfies
$\dim (X) = \dim(X^*) \leq \frac{2}{3}N$,
then $X$ is one of the followings;
(a) a hypersurface, (b) $\PP^1 \times \PP^{n-1} \subset \PP^{2n-1}$,
(c) $\Gr(1, \PP^4) \subset \PP^{9}$, (d) the $10$-dimensional spinor variety in $\PP^{15}$.
We note that the Ein's classification is completed
if \emph{Hartshorne's conjecture} holds,
that is to say, any smooth projective variety $X \subset \PN$ is a complete intersection if $\dim (X) > \frac{2}{3}N$.

If $\gamma_{n+1}$ is not generically finite, then we have $\dim(X) = \dim(X^*)$ (see \autoref{thm:eq-Zaks-formula}).
On the other hand, if $X$ is (c) or (d), then the finiteness of $\gamma_{n+1}$ can be shown by explicit calculations.
Hence 
we find that the statement of \autoref{thm:gamma_n+1-sep-bir-unlessP1Pn-1}
follows from Ein's result in characteristic $p = 0$ if $\dim(X) \leq \frac{2}{3}N$ (or Hartshorne's conjecture holds).
Our purpose is to show this statement in any characteristic $p \geq 0$
without the condition of the dimension of $X$.
\\

This paper is organized as follows.
In \textsection{}2, we generalize the techniques of 
\cite{expshr} and \cite{FI} in order to study $m$-th Gauss maps in arbitrary characteristic.
In particular, we define and describe the \emph{$m$-th degeneracy map}
and prove \autoref{thm:m-th-linearity}.
In \textsection{}3, we show basic properties of $m$-th contact loci and their dimensions, i.e., $m$-th defects.
In \textsection{}4, we study the $(n+1)$-th Gauss map $\gamma_{n+1}$
in the case when $X$ is smooth and $\gamma_{n+1}$ is not generically finite. We fix a general point $x \in X$ and consider the union of contact loci of $(n+1)$-planes $L$'s such that $L$ is tangent to $X$ at $x$.
Showing that the union is in fact an $(n-1)$-plane,
we have that $X$ is the Segre embedding $\PP^1 \times \PP^{n-1}$,
yielding
\autoref{thm:gamma_n+1-sep-bir-unlessP1Pn-1}.

\subsection*{Acknowledgments}
The authors would like to thank
Professors Satoru Fukasawa and Hajime Kaji for their valuable comments and advice.
The first author was supported by the Grant-in-Aid for JSPS fellows, No.\ 16J00404.
The second author was supported by the Grant-in-Aid for JSPS fellows, No.\ 26--1881.

\section{Linearity of general $m$-th contact loci}

In \cite[Definition 2.1]{FI},
we extended the notion
of the shrinking map, which had been
independently introduced by
J.~M.~Landsberg and J.~Piontkowski
(see \cite[2.4.7]{FP} and \cite[Theorem~3.4.8]{IL}).
Here we recall the definition of the map,
which is a main ingredient of this section.

\begin{defn}\label{thm:def-shr-Zf}
  We denote by $\sS = \sS_{m}$ and $\sQ = \sQ_{m}$ the universal subbundle of rank $N-m$ and the universal quotient bundle of rank $m+1$ over $\Gm$ respectively, 
  with the exact sequence
  \begin{equation}\label{eq:Euler0}
    0 \rightarrow \sS \rightarrow H^0(\PN, \sO(1))\otimes \sO_{\Gm} \rightarrow \sQ \rightarrow 0.
  \end{equation}

  For a rational map $f: Z \dashrightarrow \Gr(m, \PN)$,
  we define the \emph{shrinking map of $Z$ with respect to $f$}
  \[
  \sigma = \sigma_{Z, f} : Z \dashrightarrow \Gr(m^{-}, \PN)
  \]
  for some integer $m^{-} = m^{-}_{\sigma} \leq m$ as follows.
  Let $\Zo$ be an open subset consisting of smooth points of $Z$ and regard $f$ as $f|_{\Zo}$.
  We have the following composite homomorphism
  \begin{equation*}    \Phi = \Phi_f: f^*\sQ\spcheck \rightarrow f^*\sHom(\sHom(\sQ\spcheck , \sS\spcheck), \sS\spcheck)
    \rightarrow \sHom(T_{\Zo}, f^*\sS\spcheck),
  \end{equation*}
  where
  the first homomorphism
  is induced from the dual of
  $\sS \otimes \sS\spcheck \rightarrow \sO$,
  and the second one
  is induced from the differential
  $df : T_{\Zo} \rightarrow f^*T_{\GN} = f^*\sHom(\sQ\spcheck , \sS\spcheck)$.
  We define an integer $m^{-}$
  with $-1 \leq m^- \leq m$ by
  \begin{equation*}    m^{-} = \dim (\ker \Phi \otimes k(z)) - 1
  \end{equation*}
  for a general point $z \in Z$.
  Since $\ker \Phi|_{\Zo}$ is a subbundle of
  $H^0(\PN, \sO(1))\spcheck \otimes \sO_{\Zo}$
  of rank $m^{-}+1$
  (replacing $\Zo \subset Z$ by a smaller open subset if necessary),
  we have an induced morphism
  $\sigma: \Zo \rightarrow  \Gr(m^{-}, \PN)$
  and call it the \emph{shrinking map of $Z$ with respect to $f$}.
\end{defn}

The purpose of this section is to
prove the following result,
where
we write
\[
\calu_{\Gm} := \set*{(L, x) \in \Gm \times \PN}{x \in L},
\]
the universal family of $\Gm$.

\begin{thm}\label{cor_sep}
  Let $X \subset \PN$ be an $n$-dimensional projective variety,
  and let $\iota: Y \hookrightarrow \G(m,\PN)$ be a projective variety with $n \leq m \leq N-1$.
  We   set $\sigma_Y = \sigma_{Y, \iota}: Y \dashrightarrow {\Gr(m^{-},\PN)}$ to be the shrinking map of $Y$ with respect to $\iota$ , where $m^{-} = m^{-}_{\sigma_Y} \leq m$.
  Then the following are equivalent:

  \begin{enumerate}
  \item
    $\gamma_m: P_m^{X} \arw \G(m,\P^N)$ is separable,
    and $Y =\gamma_m( P_m^{X})$.

  \item $P_m^{X} = \sigma_Y^*{\calu}_{\Gr(m^{-},\PN)}$ in $\G(m,\P^N) \times \PN$.

  \item $\dim \sigma_Y^*{\calu}_{\Gr(m^{-},\PN)} =n+(m-n)(N-m)$ and the second projection $\sigma_Y^*{\calu}_{\Gr(m^{-},\PN)} \rightarrow \PN$
    is separable,
    and the image is equal to $X$.
  \end{enumerate}
\end{thm}

Once the above theorem is shown,
\autoref{thm:m-th-linearity} follows from the implication ``(1) $\Rightarrow$ (2)'' as follows.
\begin{proof}[Proof of \autoref{thm:m-th-linearity}]
  Take a general $m$-plane $[L] \in X_m^* = \gamma_m(P_m^X) = Y$
  and set $M \subset \PN$ to be the $m^{-}$-plane corresponding to $\sigma_{Y}(L) \in \Gr(m^{-},\PN)$.
  By \autoref{cor_sep} (2),
  we have $ \gamma_m^{-1}(L) =  \set{[L]} \times M$.
  Hence it holds that $\pi(\gamma_m^{-1}(L)) = M$.
\end{proof}

Note that we also have a characterization of images of separable $m$-th Gauss maps by the equivalence ``(1) $ \Leftrightarrow $ (3)'' (see \cite[Corollary 3.15]{expshr}).

\begin{rem}\label{thm:m-1-th-sep}
  If $\gamma_m$ is separable, then
  so is $\gamma_{m'}$ for all $m'$ with $n \leq m' \leq m$.
  This is shown as follows, which is the same way
  for the reflexive case by Kleiman and Piene \cite[pp.~108--109]{KP}.
  It is sufficient to show the case $m' = m-1$.
  Consider the Flag variety
  \[
  \F(m-1,m;\PN) = \set*{(M',M) \in \Gr(m-1, \PN) \times \Gr(m, \PN)}{M' \subset M},
  \]
  and define $P_{m-1,m}^X$ to be the closure of the set of
  $(M',M, x) \in \F(m-1,m;\PN) \times X^{sm}$
  such that $\TT_xX \subset M' \subset M$.
  We have the following commutative diagram:
  \[
  \xymatrix{    & P_m^X \ar[r]^{\kern-3em\gamma_m} \ar[ld]_{\pi}& \Gr(m, \PN)
    \\
    X &P_{m-1,m}^X \ar[r]^{\kern-3em\gamma_{m-1,m}} \ar[u] \ar[d] \ar[l]_{\pi}& \F(m-1,m;\PN) \ar[u] \ar[d]
    \\
    & P_{m-1}^X \ar[r]^{\kern-3em\gamma_{m-1}} \ar[lu]_{\pi}& \Gr(m-1, \PN),
  }  \]
  where $\pi(\gamma_{m-1,m}^{-1}(M',M)) = \pi(\gamma_{m-1}^{-1}(M'))$
  in $X$ for $(M',M) \in \F(m-1,m;\PN)$.
  By the diagram,
  $\pi(\gamma_{m-1}^{-1}(M')) \subset X$ is scheme-theoretically equal to
  the intersection of $\pi(\gamma_{m}^{-1}(M))$'s with
  general $m$-planes $M$ containing $M'$.
  If $\gamma_m$ is separable, then \autoref{thm:m-th-linearity} implies that $\pi(\gamma_{m}^{-1}(M))$ is a linear variety, and hence so is $\pi(\gamma_{m-1}^{-1}(M'))$.
  In particular, $\pi(\gamma_{m-1}^{-1}(M')) \simeq \gamma_{m-1}^{-1}(M')$ is reduced, that is to say $\gamma_{m-1}$ is separable.
\end{rem}

The case of $m=n \ (= \dim(X))$ in \autoref{cor_sep} is nothing but
\cite[Theorem~3.1]{expshr}, \cite[Corollary~5.4]{FI}.
In order to prove \autoref{cor_sep} for any $m \geq n$,
we generalize several techniques of \cite{FI}
in the following subsections.

\subsection{Subvarieties of the universal family of a Grassmann variety}

We set the bundles $\sS = \sS_{m}$ and $\sQ = \sQ_{m}$ on $\Gm$
as in \autoref{thm:def-shr-Zf}.
Set $\calu = \calu_{\Gm}$ to be
the universal family of $\Gm$.
Then $P_m^X \subset \calu$ holds for a projective variety $X \subset \PN$.
We denote by $\sO_{\sU}(1)$ the tautological invertible sheaf on $\sU = \cP(\sQ\spcheck)$.

First we give equivalent conditions when a subvariety of $\calu$
is contained in $P_m^X$ for some $X$.

\begin{prop}\label{thm_chara_Pmx}
  Let $X' \subset \calu \subset \G(m,\P^N) \times \P^N$ be a projective variety and let $\pr_i$ be the projection from $X'$ to the $i$-th factor for $i=1,2$.
  For simplicity,
  we also denote by $\pr_i$ the restricted morphism $\pr_i|_{\Xdo}$
  for a non-empty open subset $\Xdo \subset {X'}^{sm}$.
  Let $\Phi = \Phi_{\pr_1}: \pr_1^*\sQ\spcheck \rightarrow \sHom(T_{\Xdo}, \pr_1^*\sS\spcheck)$ and $\sigma = \sigma_{X', \pr_1}: X' \dashrightarrow \Gr(m^{-}, \PN)$ be as in \autoref{thm:def-shr-Zf}.
  
  Assume that $\pr_2 : X' \arw \P^N$ is separable.   Then the following conditions are equivalent.
  \begin{enumerate}[\normalfont(i)]
  \item\label{thm:GaussImage:graph}  $X' \subset P_m^{X}$ for $X=\pr_2(X')$.   \item\label{thm:GaussImage:zero} The composition of 
    \[
    \sO_{\calu}(-1) |_{{X'}\spcirc} \hookrightarrow \pr_1^* \sQ^{\vee}
    \ \ \text{and}\ \ \,
    \Phi : \pr_1^* \sQ^{\vee} \arw \sHom(T_{{X'}\spcirc}, \pr_1^* \sS^{\vee})
    \]
    is the zero map on a non-empty open subset ${X'}\spcirc \subset X'$.
  \item\label{thm:GaussImage:shr} The image of $(\sigma, \pr_2) : X' \dashrightarrow \Gr(m^{-}, \PN) \times \PN$
    is contained in 
    the universal family $\sU_{\Gr(m^{-}, \PN)} \subset \Gr(m^{-}, \PN) \times \PN$.
  \end{enumerate}
\end{prop}

\begin{proof}
  
  On $\calu$,
  there exists a natural surjection
  $\ep : \overline{\pr}_2^* T_{\P^N}(-1) \arw \overline{\pr}_1^* \sS^{\vee}$,
  where $\overline{\pr}_i$ is the projection from $\calu \subset \Gm \times \PN$ to the $i$-th factor
  (hence $\pr_i = \overline{\pr}_i|_{X'}$).
  We note that $\sO_{\sU}(1) = \overline{\pr}_2^*\sO_{\PN}(1)$ holds.

  Set $X=\pr_2(X')$.
  Since $\pmx$ is the incidence variety,
  $X'$ is contained in $\pmx$ if and only if 
  \[
  \pr_2^* T_{X^{sm}}(-1) \hookrightarrow \pr_2^* T_{\P^N}(-1) \arw \pr_1^* \sS^{\vee}
  \]
  is the zero map on $\Xdo$.
  By the separability of the dominant morphism $\pr_2 : \Xdo \rightarrow X \subset \P^N$,
  $d (\pr_2) : T_{\Xdo} \arw \pr_2^* T_{X^{sm}}$ is surjective (by taking smaller $\Xdo$ if necessary).
  Hence the above map is the zero map if and only if so is
  \begin{align}\label{eq_condition2}
    T_{\Xdo}(-1) \stackrel{d (\pr_2) (-1) }{\longrightarrow} \pr_2^* T_{X^{sm}}(-1) \hookrightarrow \pr_2^* T_{\P^N}(-1) \arw \pr_1^* \sS^{\vee} .
  \end{align}

  By the same argument as in the proof of (i) $\Leftrightarrow$ (ii) in  \cite[Theorem~2.4]{FI},
  we see that 
  \ref{eq_condition2} corresponds to
  \[
  \sO_{\calu}(-1) |_{{X'}\spcirc} \hookrightarrow \pr_1^* \sQ^{\vee} \stackrel{\Phi}{\longrightarrow} \sHom(T_{{X'}\spcirc}, \pr_1^* \sS^{\vee}).
  \]
  under the identification
  \[
  \sHom (  T_{\Xdo}(-1) ,  \pr_1^*\sS\spcheck) = \sHom ( \sO_{\calu}(-1)|_{\Xdo}, \sHom(T_{{X'}\spcirc}, \pr_1^* \sS^{\vee})).
  \]
  Thus (i) and (ii) are equivalent.   
  The proof of the equivalence of (ii) and (iii) is the same as that of \cite[Theorem~2.4]{FI}
\end{proof}

\subsection{Shrinking maps and $m$-th degeneracy maps}

We define the $m$-th degeneracy map as follows, which will play a key role
in the proof of \autoref{cor_sep}.

\begin{defn}\label{thm:def-kappa} 
  Let $X \subset \P^N$ be an $n$-dimensional projective variety
  and 
  $\Xo \subset X^{sm}$ be a non-empty open subset. 
  Set $\pmxo := \pmx \cap (\Gm \times \Xo)$.

  Take
  $d \gamma_m : T_{\pmxo} \rightarrow \gamma_m^{*} T_{\G(m,\P^N)}$
  and
  $d \pi: T_{\pmxo} \rightarrow \pi^{*} T_{X^{sm}}$
  to be the differentials of the $m$-th Gauss map $\gamma_m$
  and the projection $\pi: P_m^X \rightarrow X \subset \PN$.
  Then
  $d\pi (\ker (d \gamma_m)) \subset \pi^*T_{X^{sm}} \subset \pi^*T_{\PN}$
  is isomorphic to $\ker (d \gamma_m)$
  since $ T_{\pmxo} $ is contained in $T_{\G(m,\P^N) \times \P^N} |_{\pmxo}= \gamma_m^{*} T_{\G(m,\P^N)} \oplus \pi^*T_{\PN} $.

  Set $n^-_{\kappa_m} $ to be the rank of the torsion free sheaf
  $\ker (d \gamma_m) $.
  Pulling back the Euler sequence
  \[
  0 \rightarrow \sO_{\PN} \rightarrow H^0(\PN, \sO(1))\spcheck \otimes \sO_{\PN}(1) \rightarrow T_{\PN} \rightarrow 0,
  \]
  we have $\theta: H^0(\PN, \sO(1))\spcheck \otimes \sO_{\pmxo} \rightarrow \pi^*T_{\PN}(-1)$.
  Then the subsheaf
  \[
  \theta^{-1} (d\pi (\ker (d \gamma_m))(-1))
  \subset
  H^0(\P^N,\sO(1))^{\vee} \otimes \sO_{\pmxo}
  \]
  induces a rational map
  \[
  \kappa_m : P_m^{X} \dashrightarrow \G(n^-_{\kappa_m} , \P^N ),
  \]
  which we call the $m$-th \emph{degeneracy map} of $X$.
  In the case of $m=n$,
  $\kappa_n$ can be identified with the usual degeneracy map
  under the birational map $P_n^{X} \rightarrow X$.
\end{defn}

The rest of this section is devoted to the proof of the following proposition,
which is a generalization of \cite[Proposition 5.2]{FI}.

\begin{prop}\label{prop:m-th_kap-defby-shr}
  Let $X \subset \PN$ be a projective variety.
  Then the $m$-th \emph{degeneracy map} $ \kappa_m $ coincides with the shrinking map $\sigma_{P_m^{X},\gamma_m}$ of $P_m^X$ with respect to $\gamma_m: P_m^X \rightarrow \Gm$.
\end{prop}

First we recall
the definition of the second fundamental form $\tau$.
Let $\sS_{n}$ and $\sQ_{n}$ be
the universal subbundle and the universal quotient bundle over $\Gn$.
Considering the $n$-th Gauss map (ordinary Gauss map)
$\gamma_n: \Xo \rightarrow \Gn$
and considering the dual of the sequence \autoref{eq:Euler0} for $m=n$,
we have the following commutative diagram with exact rows and columns,
\begin{equation}\label{eq:Euler}
\begin{aligned}
  \xymatrix{
    & 0 \ar[d] & 0 \ar[d]
    \\
    & \sO_{\Xo} \ar[d] \ar@{=}[r]& \sO_{\Xo} \ar[d]
    \\
    0 \ar[r] & 
    \gamma_n^*\sQ_{n}\spcheck(1) \ar[r] \ar[d] &  H^0(\PN, \sO(1))\spcheck \otimes \sO_{\Xo}(1) \ar[d] \ar[r] & \gamma_n^*\sS_{n}\spcheck(1) \ar@{=}[d] \ar[r] & 0
    \\
    0 \ar[r] & T_{\Xo} \ar[r] \ar[d]
    & T_{\PN} |_{\Xo} \ar[r] \ar[d]& N_{\Xo/\PP^N}  \ar[r] & 0
    \\
    & 0 & 0 \makebox[0mm]{ $,$}
  }
  \end{aligned}
\end{equation}
where the middle vertical sequence is induced from the Euler sequence on $\PN$.
The differential
$d\gamma_n: T_{\Xo} \rightarrow \gamma_n^*T_{\Gn} = \gamma_n^*\sHom(\sQ_{n}\spcheck, \sS_{n}\spcheck)$ induces a homomorphism
\[
\widetilde{d{\gamma_n}}: T_{\Xo} \otimes \gamma_n^*\sQ_{n}\spcheck \rightarrow \gamma_n^*\sS_{n}\spcheck = N_{\Xo/\P^N}(-1).
\]
Then we can check that,
the composition of
$T_{\Xo} \otimes \sO_{\Xo} \hookrightarrow T_{\Xo} \otimes \gamma_n^*\sQ_{n}\spcheck(1)$ induced by the first vertical sequence of \ref{eq:Euler}
and
$\widetilde{d{\gamma_n}}(1): T_{\Xo} \otimes \gamma_n^*\sQ_{n}\spcheck(1) \rightarrow N_{\Xo/\P^N}$
is the zero map. 
Hence a homomorphism, called \emph{the second fundamental form},
\[
\tau: T_{\Xo} \otimes T_{\Xo} \rightarrow  N_{\Xo/\P^N}
\]
is induced.
By definition, $\widetilde{d{\gamma_n}}(1)$ factors through $\tau$.
In other words,
$d \gamma_n $ factors as
\[
T_{\Xo} \rightarrow  \sHom( T_{\Xo}(-1) ,  \gamma_n^* \sS_{n}\spcheck ) \hookrightarrow  \gamma_n^*\sHom(\sQ_{n}\spcheck, \sS_{n}\spcheck),
\]
where the inclusion is induced by the surjection $\gamma_n^*\sQ_{n}\spcheck \rightarrow T_{\Xo}(-1) $.
It is known that $\tau$
is symmetric.

\vspace{2mm}
Next we consider the following homomorphism $\phi$.
Set $\pmxo = \pmx \cap (\Gm \times \Xo)$ as before.
In this subsection, we hereafter write $\pr_i$
to be the projection from
$P_m^X \subset \Gm \times X$ to the $i$-the factor.
Indeed, $\pr_1$ and $\pr_2$ are nothing but
``$\gamma_m$'' and ``$\pi$'', respectively.
But we use the symbol ``$\pr_i$'' to suit the notation of \cite{FI}.
For simplicity of notation,
we also use $\pr_i$ for the restricted morphism
$\pr_i|_{\pmxo}$.

Write $\sS = \sS_{m}$ and $\sQ = \sQ_{m}$.
Pulling back $\tau \otimes \sO(-1)$ by $\pr_2 : P_m^X \arw X$
and composing with the natural homomorphism $\pr_2^* N_{\Xo/\P^N}(-1) \arw \pr_1^* \sS^{\vee}$,
we have the homomorphism on $\pmxo$: \[
\varphi : \pr_2^* (T_{\Xo} \otimes T_{\Xo} (-1)) \arw \pr_2^* N_{\Xo/\P^N} (-1) \arw  \pr_1^* \sS^{\vee}.
\]
We define $\varphi_i: \pr_2^* T_{\Xo}  \rightarrow  \sHom(\pr_2^* T_{\Xo}(-1),  \pr_1^*\sS\spcheck) $ for $i=1,2$ by
\begin{align*}
  \varphi_1(x) &= [\pr_2^* T_{\Xo} (-1)   \rightarrow  \pr_1^*\sS\spcheck: y \mapsto \varphi(x, y)],\\
  \varphi_2(x) &= [\pr_2^* T_{\Xo} (-1)  \rightarrow  \pr_1^*\sS\spcheck: y \mapsto \varphi(y, x)].
\end{align*}
We sometimes regard $\varphi_i$ as a homomorphism
\[
\pr_2^* T_{\Xo} \rightarrow \sHom( \pr_2^* \gamma_n^* \sQ_n\spcheck ,  \pr_1^*\sS\spcheck)
\]
by the natural surjection $ \gamma_n^* \sQ_n \spcheck \rightarrow  T_{\Xo}(-1)$.
By definition,
$\varphi_1$ is the homomorphism induced from
$\pr_2^* d\gamma_n: \pr_2^*T_{\Xo} \rightarrow \pr_2^*\gamma_n^*\sHom(\sQ_{n}\spcheck, \sS_{n}\spcheck)$
and $\pr_2^* \gamma_n^* \sS_n\spcheck = \pr_2^* N_{\Xo/\P^N}(-1) \arw \pr_1^* \sS^{\vee}$.
Furthermore, $\varphi_1 = \varphi_2$ holds because of the symmetry of $\tau$.

We have the following lemma in a general setting.

\begin{lem}\label{lem_comm_diagram}
  Let $n \leq m < N$ be non-negative integers
  and let $\mathbb{F}(n,m ;\P^N) \subset \Gr(n, \PN) \times \Gr(m, \PN)$ be the frag variety
  parametrizing linear subvarieties
  $\PP^{n} \subset \PP^{m}$ in $\PN$.
  Let $q_n : \mathbb{F}(n,m ;\P^N) \arw \G(n,\P^N)$ and
  $q_m : \mathbb{F}(n,m ;\P^N) \arw \G(m,\P^N)$
  be the natural projections.
  Then the following diagram on $\mathbb{F}(n,m ;\P^N)$ is commutative;
  \begin{equation}\label{eq:flag-comm}
  \begin{aligned}
    \xymatrix{
      T_{\mathbb{F}(n,m ;\P^N)}   \ar[d]_(.5){d q_m} \ar[r]^(.4){d q_n}
      &    q_n^* T_{\G(n,\P^N)} =  q_n^* (\sQ_{n} \otimes  \sS_{n}\spcheck)  \ar[d]
      \\
      q_m^* T_{\G(m,\P^N)} =  q_m^* (\sQ_{m} \otimes  \sS_{m}\spcheck)\ar[r]
      & q_n^* \sQ_{n} \otimes   q_m^*  \sS_{m}\spcheck ,
    }
    \end{aligned}
  \end{equation}
  where the bottom and right maps are induced by the natural homomorphisms $q_n^* \sS_{n}^{\vee} \arw q_m^* \sS_{m}^{\vee} $
  and $q_m^* \sQ_{m} \arw q_n^* \sQ_{n}$ respectively.
  
  In particular,
  we obtain a commutative diagram with exact rows
  \begin{equation}\label{diag_lem_comm}
    \begin{aligned}
      \xymatrix@C=15pt{
        0 \ar[r] & T_{\mathbb{F}(n,m ;\P^N) / \G(n,\P^N)}     \ar[r] \ar[d]_{\wr}& T_{\mathbb{F}(n,m ;\P^N)}   \ar[d]_{dq_m}\ar[r]^{dq_n}
        &    q_n^* T_{\G(n,\P^N)}  \ar[d] \ar[r]& 0
        \\
        0 \ar[r] & \mathscr{K} \otimes q_m^* \sS_{m}\spcheck \ar[r]&    q_m^* (\sQ_{m} \otimes  \sS_{m}\spcheck)\ar[r]
        & q_n^* \sQ_{n} \otimes   q_m^*  \sS_{m}\spcheck \ar[r] & 0,
      } 
    \end{aligned}
  \end{equation}
  where $\mathscr{K} = \ker   (  q_m^*  \sQ_{m} \rightarrow q_n^* \sQ_{n} ) \simeq q_n^* \sS_{n} / q_m^* \sS_{m}$.
\end{lem}

\begin{proof}
  
  The case when $n = 0$ is nothing but
  \cite[Lemma 2.6]{FI} ($m$ corresponds to ``$n$'' of that).
  As in the proof of that lemma,
  the commutativity of \autoref{eq:flag-comm} can be checked
  by taking local coordinates on $\P^N$.
  We leave the detail to the reader.
  
  We note that the left homomorphism in \autoref{diag_lem_comm} is an isomorphism since $\mathbb{F}(n,m ;\P^N)$ is nothing but the Grassmann bundle $G(N-m , \sS_{n} \spcheck)$ over $ \G(n,\P^N)$.
\end{proof}

Now return to the original setting.
The sheaves
\[
\pr_2^* (\gamma_n^*\sQ_{n}\spcheck) \subset \pr_1^* \sQ^{\vee} \subset H^0 (\P^N,\sO(1))^{\vee} \otimes \sO_{\pmxo}
\]
on $\pmxo \subset \Gm \times \Xo$ define a morphism 
\[
(\gamma_n \circ \pr_2 , \pr_1) :  \pmxo \arw \mathbb{F}(n,m ;\P^N) \quad : \quad (x,L) \mapsto (\TT_x X, L).
\]
We note that $\pmxo$ is the fiber product of $\Xo$ and $\mathbb{F}(n,m ;\P^N) $ over $ \G(n,\P^N)$ by the diagram
\begin{equation}\label{diag_fiber_product}
  \begin{aligned}
    \xymatrix{
      \pmxo  \ar[d]_{\pr_2} \ar[rr]^{\kern-2em(\gamma_n \circ \pr_2 , \pr_1) }
      & &  \mathbb{F}(n,m ;\P^N)   \ar[d]
      \\
      \Xo  \ar[rr]^{\gamma_n}
      & & \G(n,\P^N).
    }   \end{aligned}
\end{equation}

\begin{lem}\label{calim_degeneracy_map}
  The kernel of the differential
  $d \pr_1 : T_{\pmxo} \arw \pr_1^* T_{\G(m,\P^N)}$
  is isomorphic to the kernel of $\varphi_1$
  under $d \pr_2: T_{\pmxo} \rightarrow \pr_2^* T_{\Xo}$.
\end{lem}

\begin{proof}
  By pulling back the diagram \autoref{diag_lem_comm} by $(\gamma_n \circ \pr_2 , \pr_1)  $,
  we have a commutative diagram with exact rows
  \[
  \xymatrix{   0 \ar[r] &T_{\pmxo /\Xo} \ar[r] \ar[d]_{\wr} &  T_{\pmxo}  \ar[r]^(.35){d \pr_2} \ar[d]_{d \pr_1}
    &  \pr_2^* T_{\Xo}  \ar[d]^{ \varphi_1} \ar[r]&0
    \\
    0 \ar[r] & (\gamma_n \circ \pr_2 , \pr_1)^* \mathscr{K} \otimes   \pr_1^* \sS\spcheck  \ar[r] &  \pr_1^* (\sQ \otimes  \sS\spcheck)   \ar[r]
    &  \pr_2^* ( \gamma_n^*\sQ_{n}) \otimes   \pr_1^*  \sS\spcheck  \ar[r] &0
  } 
  \]
  on $\pmxo \subset \Gm \times \Xo$.
  We note that the left homomorphism is an isomorphism
  since the diagram \autoref{diag_fiber_product} is Cartesian.
  Hence $\ker d \pr_1 \simeq \ker \varphi_1$ holds.
\end{proof}

\begin{lem}\label{calim_Phi}
  For $\Phi = \Phi_{\pr_1}$,
  $\ker \Phi \subset \pr_1^* \sQ^{\vee} \subset H^0(\PN, \sO(1))\spcheck \otimes \sO_{\pmxo}$ coincides with 
  the kernel of the composite homomorphism 
  \[
  \Phi' \ : \  \pr_2^*(\gamma_n^*\sQ_{n}\spcheck) \arw \pr_2^* T_{\Xo}(-1) \stackrel{\varphi_ 2(-1)}{\longrightarrow}  \pr_2^* \Omega_{\Xo} \otimes  \pr_1^*\sS\spcheck
  \]
  as a subsheaf of $ H^0(\PN, \sO(1))\spcheck \otimes \sO_{\pmxo}$.
\end{lem}

\begin{proof}
  The diagram \autoref{diag_lem_comm} induces
  \[
  \xymatrix{
    0 \ar[r] & q_n^* \sQ_{n}\spcheck \ar[r] \ar[d] &  q_m^* \sQ\spcheck   \ar[r] \ar[d]
    &  \mathscr{K}\spcheck  \ar@{^(->}[d] \ar[r]&0
    \\
    0 \ar[r] & q_n^* \Omega_{\G(n,\P^N)} \otimes   q_m^* \sS\spcheck  \ar[r] & \Omega_{\mathbb{F}(n,m ;\P^N)} \otimes   q_m^* \sS\spcheck    \ar[r]
    &  \Omega \otimes   q_m^* \sS\spcheck \ar[r] &0
  } 
  \]
  on $\mathbb{F}(n,m ;\P^N)$, 
  where   $\Omega = \Omega_{\mathbb{F}(n,m ;\P^N)/ \G(n,\P^N)} \simeq \mathscr{K}\spcheck \otimes  q_m^* \sS$.
  Since the natural homomorphism $\sO_{\mathbb{F}(n,m ;\P^N)} \rightarrow  q_m^* (\sS \otimes \sS\spcheck)$ is injective,
  so is the right vertical homomorphism.

  By pulling back the above diagram by $(\gamma_n \circ \pr_2 , \pr_1)  $,
  we obtain
  \[
  \xymatrix{   0 \ar[r] & \pr_2^*(\gamma_n^*\sQ_{n}\spcheck) \ar[r] \ar[d]^{\Phi'} &  \pr_1^* \sQ\spcheck   \ar[r] \ar[d]^{\Phi}
    &  (\gamma_n \circ \pr_2 , \pr_1)^* \mathscr{K}\spcheck  \ar@{^(->}[d] \ar[r]&0
    \\
    0 \ar[r] & \pr_2^* \Omega_{T_{\Xo}} \otimes   \pr_1^* \sS\spcheck  \ar[r] & \Omega_{\pmxo} \otimes   \pr_1^* \sS\spcheck    \ar[r]
    &  \Omega_{\pmxo/ \Xo} \otimes   \pr_1^* \sS\spcheck \ar[r] &0,
  } 
  \]
  on $\pmxo \subset \Gm \times \Xo$.
  We note that the middle homomorphism is nothing but $\Phi = \Phi_{\pr_1}$
  and the left one is $\Phi'$ in the statement of this lemma.
  Hence $\ker \Phi $ coincides with $\ker \Phi'$ as a subsheaf of $ H^0(\PN, \sO(1))\spcheck \otimes \sO_{\pmxo}$.
\end{proof}

\begin{proof}[Proof of \autoref{prop:m-th_kap-defby-shr}]
  By \autoref{calim_Phi},
  $\sigma_{P_m^X, \gamma_m}$ is induced by the kernel of the composition
  $\pr_2^*(\gamma_n^*\sQ_{n}\spcheck) \arw \pr_2^* T_{\Xo}(-1) \stackrel{\varphi_2(-1)}{\longrightarrow}  \pr_2^* \Omega_{\Xo} \otimes  \pr_1^*\sS\spcheck$.
  By \autoref{calim_degeneracy_map},
  $\kappa_m$ is induced by the kernel of
  $\pr_2^*(\gamma_n^*\sQ_{n}\spcheck) \arw \pr_2^* T_{\Xo}(-1) \stackrel{\varphi_1(-1)}{\longrightarrow}  \pr_2^* \Omega_{\Xo} \otimes  \pr_1^*\sS\spcheck$.
  Since $\varphi_1=\varphi_2$,
  this proposition follows.
\end{proof}

\begin{proof}[Proof of \autoref{cor_sep}]
  (1) $\Rightarrow$ (2):  
  Let $\kappa_m: P_m^X \dashrightarrow \G(n^-_{\kappa}, \P^N)$
  be the $m$-th degeneracy map,
  where $n^-_{\kappa} =\rank ( \ker d \gamma_m )  $ for $d \gamma_m : T_{\pmxo} \arw \gamma_m^* T_Y$.
  Since $d \gamma_m$ is surjective,
  we have $n^-_{\kappa} = \dim \pmx- \dim Y$.
  From \autoref{prop:m-th_kap-defby-shr}, we have
  $\kappa_m = \sigma_{P_m^X,\gamma_m}$.
  Since $\gamma_m$ is separable, \cite[Remark 2.3]{FI} implies
  $\sigma_{P_m^X,\gamma_m} =\sigma_Y \circ  \gamma_m$; hence
  $\kappa_m = \sigma_Y \circ  \gamma_m$ and $n^-_{\kappa} = m^{-}_{\sigma_Y} = m^{-}$ hold.
  Thus we have $m^-=\dim P_m^X- \dim Y$.

  By \autoref{thm_chara_Pmx},
  it holds that $P_m^X \subset \sigma_Y^{*} \calu_{\G(m^-,\P^N)}$.
  Since 
  \[
  \dim \sigma_Y^{*} \calu_{\G(m^-,\P^N)} = \dim Y + m^- = \dim P_m^X,
  \]
  $P_m^X$ coincides with $\sigma_Y^{*} \calu_{\G(m^-,\P^N)}$.

  (3) $ \Rightarrow$ (2):
  Set $X' = \sigma_Y^{*} \calu_{\G(m^-,\P^N)}$.
  Since $\pr_1 : X' \arw Y$ is a projective bundle,
  it is separable;
  hence $\sigma_{X',\pr_1} =\sigma_Y \circ \pr_1$ holds and then the condition (iii) of \autoref{thm_chara_Pmx} is satisfied.
  Thus \autoref{thm_chara_Pmx} implies (2).

  The implications (2) $\Rightarrow$ (1) and (2) $\Rightarrow$ (3)
  follow immediately.
\end{proof}

As we have already seen,
\autoref{cor_sep} implies \autoref{thm:m-th-linearity}.
In fact, we have the following result.

\begin{cor}\label{thm:sigma_m-eq-cloc}
  Assume that $\gamma_m$ is separable.
  Let $\sigma_{X_{m}^*} = \sigma_{X_{m}^*, \iota}: X_{m}^* \dashrightarrow \Gr(m^{-}, \PN)$
  be the shrinking map of $X_m$ with respect to $\iota: X_{m}^* \hookrightarrow \Gm$.
  Then $\sigma_{X_{m}^*}(L) \in \Gr(m^{-}, \PN)$
  corresponds to the contact locus $\pi(\gamma_m^{-1}(L)) \subset X$
  for a general tangent $m$-plane $L \in X_m^*$.
\end{cor}

\begin{proof}
  As in the proof of \autoref{thm:m-th-linearity},
  taking $Y = X_{m}^*$, we have the assertion
  from \autoref{cor_sep} (2).
\end{proof}

\section{Properties of $m$-th defects}

Let $X \subset \PN$ be a non-degenerate projective variety of dimension $n$.
For an integer $m$ with $n  \leq m < N$,
we write $\delta_{m} = \delta_{m}(X) := \dim \pmx - \dim X_m^*$, the \emph{$m$-th defect} of $X$.
In this section,
we do not assume the separability of $\gamma_m$.

We set $\sigma_m^{L} := \pi(\gamma_m^{-1}(L)_{red}) \subset X$,
the contact locus \autoref{eq:conloc} of an tangent $m$-plane $L \in X_{m}^*$.
Then $\delta_m = \dim(\sigma_m^{L})$ for general $L \in X_{m}^*$.
Note that $\sigma_m^{L} \subset X$ is equal to $\sigma_{X_m^*}(L)$ if $\gamma_m$ is separable (see \autoref{thm:sigma_m-eq-cloc}).

E.~Ballico showed the following statement
for reflexive $X$ (see \cite[Proposition 1]{Ballico}).
We show it in any case.

\begin{lem}\label{thm:del_m-inc}
  It holds that
  $\delta_{m-1}+\delta_{m+1} \geq 2 \delta_{m}$ for $n < m < N-1$.
\end{lem}

\begin{proof}

  Let $\F(m_1,\ldots,m_r;\P^N)$ be the flag variety parametrizing
  $\P^{m_1} \subset \P^{m_2} \subset \cdots \subset \P^{m_r} \subset \P^N$,
  and let
  \begin{align*}
    \calv &:= \F(m-1,m,m+1 ; \P^N) \times_{ \F(m-1, m+1 ; \P^N)}  \F(m-1,m,m+1 ; \P^N) \\
    &=\{ (M' ,L_1,L_2,M) \, | \, M' \subset L_i \subset M \text{ for } i=1,2 \} \\
    &\subset \G(m-1,\P^N) \times \G(m,\P^N) \times \G(m,\P^N) \times \G(m+1,\P^N) .
  \end{align*}
  Then
  \[
  \calv \times_{\G(m-1,\P^N)} P_{m-1}^X =\overline{ \{ (M' ,L_1,L_2,M, x) \, | \, x \in X^{sm} \text{ and }\T_x X \subset M'  \} }
  \]
  is irreducible.
  Take a general $(M',L_1,L_2,M, x) \in  \calv \times_{\G(m-1,\P^N)} P_{m-1}^X $.
  Then $( M', x) , (L_i, x), (M, x)  $ are general in $P_{m-1}^X, P_{m}^X, P_{m+1}^X$ respectively
  since projections
  \begin{align*}
    \calv \times_{\G(m-1,\P^N)} P_{m-1}^X  &\arw P_{m-1}^X : ( M',L_1,L_2,M, x) \mapsto ( M', x)   \\
    \calv \times_{\G(m-1,\P^N)} P_{m-1}^X  &\arw P_{m}^X : ( M',L_1,L_2,M, x) \mapsto  (L_i, x)  \\
    \calv \times_{\G(m-1,\P^N)} P_{m-1}^X  &\arw P_{m+1}^X : ( M',L_1,L_2,M, x) \mapsto  (M, x)  
  \end{align*}
  are surjective.

  Hence $\sigma_{m-1}^{M'}, \sigma_m^{L_i}, \sigma_{m+1}^{M}$ are smooth at $x$ of dimensions $\delta_{m-1}, \delta_m,\delta_{m+1}$ respectively.
  Since $\sigma_m^{L_i}$ is of codimension $r = \delta_{m+1} - \delta_{m}$ in $\sigma_{m+1}^{M}$ at $x$,
  the dimension $\delta_{m-1}$ of the intersection $ \sigma_{m-1}^{M'}=\sigma_m^{L_1} \cap \sigma_m^{L_2} $ at $x$ is at least $\geq \delta_m - r$.
\end{proof}

\begin{rem}
  It is known that the property of \autoref{thm:del_m-inc}
  induces the \emph{convexity}, that is,
  for $4$ integers $a,b,c,d$ with $a+d = b+c$ and $n \leq a < b \leq c < d \leq N-1$, it holds that
  $\delta_{a}+\delta_{d} \geq \delta_{b}+\delta_{c}$.
\end{rem}

For smooth $X$,
since $\delta_n = 0$,
the convexity 
gives the following corollaries.

\begin{cor}\label{thm:del_m-inc-if-pos}\label{thm:decrease-1}
  Assume that $X$ is smooth,
  and let $m$ be an integer with $n < m < N-1$.
  Then the following holds.
  \begin{enumerate}
  \item 
    If $\delta_m > 0$,
    then $\delta_{m+1} > \delta_m$.
  \item 
    Assume $\delta_{m} > 0$ and $\delta_{m+1} = \delta_{m} + 1$.
    Then $\delta_{m-1} = \delta_{m} - 1$.
    Moreover $\delta_{m-i} = \delta_{m} - i$ holds for all $0 \leq i \leq \delta_{m}$.
  \end{enumerate}
\end{cor}

By a theorem of Zak \cite[I, 2.3 Theorem]{Zak},
if $X \subset \PN$ is smooth, then $\delta_{n+i} \leq i$ for
all integer $i \geq 0$ with $n+i \leq N-1$. Hence we have:

\begin{cor}\label{thm:eq-Zaks-formula}
  Assume that $X$ is smooth,
  and assume that $\delta_{n+{i_0}} = i_0$ holds for an integer $i_0 > 0$ with $n+i_0 < N-1$.
  Then $\delta_{n+i} = i$ holds for all integer $i \geq 0$ with $n+i \leq N-1$.
  In particular, $\delta_{N-1} = N-n-1$, that is, $\dim(X^*) = n$.
\end{cor}
\begin{proof}
  It follows for $i > i_0$ due to Zak's theorem
  and \autoref{thm:del_m-inc-if-pos} (1).
  It follows for $i < i_0$ by
  \autoref{thm:del_m-inc-if-pos} (2).
\end{proof}

\begin{lem}\label{thm:del-geq-kn}
  Let $X \subset \PN$ be a projective variety of dimension $n$,
  and let $k,m$ be integers with $0 < k < n < m < N$.
  Let $A \subset \PN$ be a $k$-plane contained in $X$ with $A \cap X^{sm} \neq \emptyset$,
  and let $L \subset \PN$ be an $m$-plane
  tangent to $X$ at some point of $A \cap X^{sm} $.
  Then $\sigma_m^L \cap A$ is of dimension $\geq (N-m+1)k - (N-m)n$.
  In particular,
  \[
  \delta_m \geq (N-m+1)k - (N-m)n
  \]
  if $X$ is covered by $k$-planes.
\end{lem}

\begin{proof}
  Set $A^{\circ} = A \cap X^{sm}$.
  We consider
  \[
  \xymatrix{  B^{\circ}:= \Fnm \times_{\Gn} A^{\circ} \ar[r] \ar[d]
    & \Fnm \ar[r] \ar[d] & \Gm
    \\
    A^{\circ} \ar[r]^{\gamma_n|_{A^{\circ}}} & \Gn,
  }  \]  
  where $\gamma_n|_{A^{\circ}} $ is the restriction of the Gauss map of $X$ on $A^{\circ}$.
  Since 
  \[
  B^{\circ}=\{ (L, a) \, | \, \T_a X \subset L \} \subset \G(m,\P^N) \times A^{\circ},
  \]
  we have $A \subset L$ for $(L, a) \in B^{\circ}$.
  Hence the image of
  $f: B \rightarrow \Gm$
  is contained in $A_m^* = \set*{L \in \Gm}{A \subset L}$,
  where $B:=\overline{B^{\circ}} \subset P^X_m$.
  Since $\dim(B) = k+(m-n)(N-m)$
  and $\dim f(B) \leq \dim(A_m^*) = (m-k)(N-m)$,
  each fiber of $B \rightarrow f(B)$ is of dimension
  $\geq k+(k-n)(N-m)$.
\end{proof}

\begin{ex}\label{thm:cov-by-n-1}
  Let $X \subset \PN$ be a projective bundle over a smooth curve $C$
  such that each fiber of $X \rightarrow C$ is a linear subvariety
  of $\PN$.
  It is well known that $\delta_{N-1}=n-2$.
  The inequality $\delta_{N-1} \geq n-2$ follows from \autoref{thm:del-geq-kn} since
  $X$ is covered by $(n-1)$-planes.
  The opposite inequality $\delta_{N-1} \leq n-2$ follows from \cite[Chapter I, 2.3 Theorem b)]{Zak}. 
  
  Moreover,
  $\delta_{N-i} \geq n-1-i$ holds for each $1 \leq i \leq n-1$ by \autoref{thm:del-geq-kn}.
  By $\delta_{N-1}=n-2$ and \autoref{thm:decrease-1},
  we have $\delta_{N-i} = n-1-i$ for $1 \leq i \leq n-1$.

\end{ex}

\begin{rem}\label{thm:lin-proj-m-def}
  Let $X \subset \PN$ be a non-degenerate projective variety.
  Assume that the secant variety $S(X)$ is not equal to $\PN$,
  and take a general point $z \in \PN$ with $z \notin S(X)$.
  Let $\pi_z: \PN \setminus \set{z} \rightarrow \PP^{N-1}$ be the linear projection.
  Since $z \notin S(X)$, $\pi_z|_{X}: X \rightarrow X_z := \pi_z(X) \subset \PP^{N-1}$ is isomorphic to the image $X_z$.
  In this setting, we have $\delta_m(X) = \delta_{m-1}(X_z)$ for $m > n$.
  The reason is as follows.
  
  First, we note that
  for the contact locus $A \subset X$
  of a general tangent $m$-plane $M$ with $z \in M$, we have
  $\dim(A) = \delta_m(X)$ since $z$ is general.
  
  Let $M' \subset \PP^{N-1}$ be a $(m-1)$-plane,
  and set $M := \pi_z^{-1}(M') \cup \set{z}$.
  Then $M'$ is tangent to $X_z$ at a smooth point if and only if $ M$ is tangent to $X$ at a smooth point.
  In such a case,
  the contact locus $A' \subset X_z$ of $M'$ coincides with $\pi_z(A)$,
  where $A \subset X$ is the contact locus of $M$.
  Hence if $M'$ is a general $(m-1)$-plane tangent to $X_z$,
  it holds that $\delta_{m-1}(X_z) =\dim (A') =\dim A = \delta_m (X) $.
\end{rem}

\begin{ex}
  Let $X' \subset \PP^{N'}$ be a projective bundle over a smooth curve
  $C$ such that a general fiber of $X' \rightarrow C$ is a linear variety in $\PP^{N'}$.
  Since $\dim S(X') \leq 2n+1$, we can take
  $X \subset \PP^{2n+1}$ as the image of $X'$ under $\PP^{N'} \dashrightarrow \PP^{2n+1}$, a composition of some linear projections,
  such that $X' \simeq X$.
  Then $X$ satisfies $\delta_{n+3} = 1$ as in \autoref{thm:cov-by-n-1}.
\end{ex}

\begin{ex}
  Let $X \subset \PP^{2n}$
  be a projective bundle over a smooth elliptic curve such that each fiber is a linear variety in $\PP^{2n}$ (see \cite[Corollary 2.3]{CH} for the existence of such $X$).
  Then $X$ satisfies $\delta_{n+2} = 1$ as in \autoref{thm:cov-by-n-1}.
\end{ex}

\begin{ex}\label{thm:P1Pn-1}
  Let $X = \PP^1 \times \PP^{n-1} \subset \PP^{2n-1}$, the Segre embedding.
  Then $X$ satisfies $\delta_{n+1} = 1$; moreover $\delta_{n+i} = i$ for $i \geq 0$ as in \autoref{thm:cov-by-n-1}.
\end{ex}

\section{Varieties with positive $(n+1)$-th defect}

Let $X \subset \PN$ be a non-degenerate projective variety of dimension $n$ over
an algebraically closed field in arbitrary characteristic.
In this section, we assume that $X$ is \emph{smooth}
and $n+2 \leq N-1$.
Let $X_{m,x}^* := \set*{L \in \Gr(m, \PN)}{\TT_xX \subset L} \subset X_m^*$, the set of $m$-planes tangent to $X$ at $x$.

\begin{defn}
  Let $x \in X$ and $L \in X_{m,x}^*$ be general.
  Then we can assume $x \in (\sigma_m^L)^{sm}$,
  and hence there exists a \emph{unique irreducible
    component} of $\sigma_m^L$ containing $x$,
  which we denote by $\Wm$.
  By generality, we can also assume $\dim(\Wm) = \delta_m$.
\end{defn}

Now we assume that $\gamma_{n+1}$ is separable
and $\delta_{n+1} = 1$.
Then $\sigma_{n+1,x}^L = \sigma_{n+1}^L$ is a line by
\autoref{thm:m-th-linearity}.
From \autoref{thm:eq-Zaks-formula},
we have $\delta_{n+2} = 2$.

\begin{lem}\label{thm:sigma-m+A-in-Lambda}
  Take a general $(M, x) \in P_{n+2}^X$.
  Then the unique irreducible component $\snxM$ of $\sigma_{n+2}^M$ containing $x$ is a surface covered by
  lines $\sigma_{n+1}^L$'s with general $L \in X_{n+1,x}^*$ satisfying $L \subset M$.
\end{lem}

\begin{proof}
  Take a general $L \in X_{n+1,x}^*$ with $L \subset M$.
  Since $\sigma_{n+1}^L \subset \sigma_{n+2}^M$,
  the line $\sigma_{n+1}^L$ is contained in the surface $\snxM$.
  For two general $(n+1)$-planes
  $L, L' \in X_{n+1,x}^*$ with $L, L' \subset M$,
  we have
  \[
  \sigma_{n+1}^L \neq \sigma_{n+1}^{L'}
  \]
  as follows.
  Suppose that the equality holds. 
  Since $(n+1)$-planes $L, L'$ contain $n$-plain $\TT_xX$ and $L \neq L'$,
  we have
  $L \cap L' = \TT_xX$.
  Since two lines $\sigma_{n+1}^L$ and $\sigma_{n+1}^{L'}$ coincide,
  taking a general point $x'$ of the line,
  we have $\TT_{x'}X \subset L \cap L' = \TT_xX$; thus
  $\TT_{x'}X = \TT_xX$.
  This contradicts the finiteness of $\gamma_n$ for smooth $X$.
  
  Hence two lines
  $\sigma_{n+1}^L$ and $\sigma_{n+1}^{L'}$
  are distinct, and then
  such lines cover the surface $\snxM$.
\end{proof}

We denote by $(X_{n+1}^*)\spcirc$ the set of $L \in X_{n+1}^*$
such that $\sigma_{n+1}^L$ is a line,
and by $(P^X_{n+1})\spcirc$ the intersection $P^X_{n+1} \cap (\PN \times (X_{n+1}^*)\spcirc)$.
Let us consider
\begin{equation}\label{eq:PxU-iso-0}
  \PmPm
  \simeq 
  \set*{(L,x,x') \in \Gr(n+1, \PN) \times X \times X}{\TT_xX, \TT_{x'}X \subset L},
\end{equation}
\begin{equation}\label{eq:PxU-iso}
  \PmPmo \subset \PmPm
  \rightarrow X \times X,
\end{equation}
and set
\begin{equation}\label{eq:Lambda-in-XxX}
  \Lambda\spcirc \subset X \times X
\end{equation}
to be the image of the morphism \ref{eq:PxU-iso}.
Let $\Lambda \subset X \times X$ be the closure of $\Lambda\spcirc$,
where the $i$-th projection $\rho_i: \Lambda \rightarrow X$ with $i = 1,2$ is separable
since so is $\gamma_{n+1}$.

We denote by $\Lmx = \rho_2(\rho_1^{-1}(x)) \subset X$.
As a set, $\Lmx $ can be described as follows:
For general $x \in X$,
set $(X_{n+1,x}^*)\spcirc = X_{n+1,x}^* \cap (X_{n+1}^*)\spcirc$.
Then the fiber of 
\begin{equation*}
  \PmPmo \subset \PmPm
  \rightarrow X : (L, x,x') \mapsto x
\end{equation*}
over general $x \in X$ is
\[
\{ (L, x') \in (X_{n+1,x}^*)\spcirc \times X \, | \, x' \in \sigma^L_{n+1} \} .
\]
Hence the fiber $\Lmx\spcirc$ over $x$ of the first projection
$\Lambda\spcirc  \arw X$ is 
\begin{equation*}
  \bigcup_{L \in (X_{n+1,x}^*)\spcirc} \sigma_{n+1}^L = \pi (\gamma_{n+1}^{-1} (X_{n+1,x}^*)\spcirc) \subset X
\end{equation*}
and $\Lmx$ is the closure of this set.
We note that for general $M \in X_{n+2,x}^*$,
we have $\snxM \subset \Lmx$ by \autoref{thm:sigma-m+A-in-Lambda}
since $(x,M) \in P^X_{n+2,x}$ is general.

\begin{rem}\label{thm:cone_at_x}\label{thm:l-leq}
  \begin{inparaenum}
  \item We set
    \[
    \sigma:= \sigma_{X_{n+1}^*}: X_{n+1}^* \dashrightarrow \Gr(1, \PN)
    \]
    to be the shrinking map of $X_{n+1}^*$ with respect to $\iota: X_{n+1}^* \hookrightarrow \Gr(n+1, \PN)$.
    Then $\sigma(L) = \sigma_{n+1}^L$ for general $L \in X_{n+1}^*$ as in \autoref{thm:sigma_m-eq-cloc}.

  \item 
    $\Lmx$ is an irreducible cone with vertex $x$.
    The reason is as follows.
    It is irreducible
    since
    $\Lmx\spcirc$ coincides with the image of
    \[
    \calu_{\Gr(1, \PN)}|_{\sigma((X_{n+1,x}^*)\spcirc)} \subset \Gr(1, \PN) \times X
    \]
    under the second projection.
    In addition, it is a cone with vertex $x$
    since each $\sigma_{n+1}^L$ is a line containing $x$.

  \item 
    Each fiber of $\PmPm \rightarrow P_{n+1}^X$ at $(L,x) \in P_{n+1}^X$
    corresponds to $\sigma_{n+1}^L$, whose dimension is $\delta_{n+1}=1$ if $(x,L)$ is general.
    Since $\dim (P_{n+1}^X) = N-1$, we have
    $\PmPm = N$.
  \end{inparaenum}
\end{rem}

\begin{lem}\label{thm:sig_m-fib}
  $\dim \Lambda_x = N-n$ for general $x \in X$.
\end{lem}

\begin{proof}
  Let $\sigma$ be the shrinking map as in \autoref{thm:cone_at_x}.
  Let $L_F:= \lin{\bigcup_{x \in F} \TT_xX} \subset \P^N$ be the linear variety spanned by $\bigcup_{x \in F} \TT_xX$ for general $F \in \sigma(X_{n+1}^*)$.
  Since $X$ is smooth, $\bigcup_{x \in F} \TT_xX$ is of dimension $\geq n+1$, and so is $L_F$.
  Taking $L \in X_{n+1}^*$ such that $F = \sigma(L) = \pi(\gamma_{n+1}^{-1}(L))$, we have
  $L_F \subset L$ and then $L_F = L$.
  In particular, $L_F$ is an $(n+1)$-plane.
  Therefore
  \begin{equation}\label{eq:simga-gen-inj}
    {\sigma}^{-1}(F) = \set*{L \in \Gr(n+1, \PN)}{L_F \subset L}
    = \set{L_F}.
  \end{equation}
  In particular, $\sigma$ is generically injective.

  Now we show that the morphism
  \[
  f: \PmPmo \rightarrow \Lambda\spcirc
  \]
  given by \autoref{eq:PxU-iso}
  is generically bijective, as follows.
  A point of the left hand side is expressed by $(L, x,x')$ as in \ref{eq:PxU-iso-0}.
  For general $(x,x') \in \Lambda\spcirc$,
  we take a point $(L,x,x') \in f^{-1}(x,x')$.
  Then, since $x, x'$ are points of the line $\sigma(L)$,
  we have $\sigma(L) = \overline{xx'}$.
  It means that 
  $f^{-1}(x,x') \simeq \sigma^{-1}(\overline{xx'}) \times \set{(x,x')}$, which is indeed equal to the set of
  a point $(L_{\overline{xx'}}, x, x') \in \Gr(n+1, \PN) \times X \times X$
  because of \autoref{eq:simga-gen-inj}.

  Therefore $\dim(\Lambda\spcirc)
  = \dim (\PmPmo) = N$
  (see \autoref{thm:l-leq}).
  It follows that
  a general fiber of $\Lambda\spcirc \rightarrow X$
  is of dimension $N-n$.
\end{proof}

\begin{lem}\label{thm:Lx=Lx'}
  Let $x \in X$ be a general point.
  Then $\Lmx = \Lambda_{x'}$ for general $x' \in \Lmx$.
  Therefore $\Lmx$ is scheme-theoretically a linear variety of $\PN$.
\end{lem}

\begin{proof}
  Let $x \in X$ and $x' \in \Lmx$ be general points as in the statement of this proposition.
  In other word,
  take general $(x,x') \in   \Lambda\spcirc$.
  By definition,
  there exists an $(n+1)$-plane $L$ such that
  $(L, x,x') $ is a general point in $\PmPmo$.

  Take general $K \in  X_{n+1,x}^* $ and set $M=\langle L,K \rangle$,
  where $M$ is an $(n+2)$-plane since $L \cap K = \TT_xX$.
  Then 
  $(x,M) , (x',M) \in P_{n+2}^X$ are general,
  which will be shown later.

  Hence we have the unique irreducible component $\sigma_{n+2,x}^M$ (resp.\ $\sigma_{n+2,x'}^M$)
  of $ \sigma_{n+2}^{M}$ containing $x$ (resp.\ $x'$).
  Since $L, K \subset M$,
  the lines $\sigma_{n+1}^L, \sigma_{n+1}^K$ are contained in $\sigma_{n+2}^M$.
  Hence
  $\sigma_{n+1}^L, \sigma_{n+1}^K$ must be contained in $\sigma_{n+2,x}^M$ because of $x \in \sigma_{n+1}^L, \sigma_{n+1}^K$.
  Since $x' \in  \sigma_{n+1}^L$,
  $x'$ is also contained in $\sigma_{n+2,x}^M$.
  By the uniqueness of $\sigma_{n+2,x'}^M$,
  we have $\sigma_{n+2,x'}^M =\sigma_{n+2,x}^M$.
  Hence it follows from \autoref{thm:sigma-m+A-in-Lambda} that
  $\Lambda_{x'} \supset \sigma_{n+2,x'}^M =\sigma_{n+2,x}^M \supset  \sigma_{n+1}^K$.
  Recall that $K \in  X_{n+1,x}^* $ can be any general element.
  Hence we have
  \[
  \Lambda_{x'}  \supset \overline{\bigcup_{K \in  X_{n+1,x}^* , \text{ general}} \sigma_{n+1}^K } = \Lambda_{x}.
  \]
  Since $\Lambda_{x'}, \Lambda_{x}$ have the same dimension,
  $\Lambda_{x'}= \Lambda_{x}$ holds.

  As a result,
  $\Lmx$ is a cone with vertex $x'$ for general $x' \in \Lmx$,
  because of $\Lmx = \Lambda_{x'}$ and \autoref{thm:cone_at_x}.
  This implies that $\Lmx$ is a linear variety.
  We note that $\Lambda_x$ is reduced since $\rho_1 : \Lambda \rightarrow X$ is separable.

  \vspace{2mm}
  To finish the proof,
  it suffices to check that for general $(L, x,x') \in \PmPmo$ and general $K \in  X_{n+1,x}^* $,
  $(\langle L,K \rangle, x) , (\langle L,K \rangle, x') \in P_{n+1}^X$
  are also general elements.

  To parametrize $(L, x,x') \in \PmPmo$ and $K \in  X_{n+1,x}^* $,
  consider
  \begin{align*}
    P^X_{n+1} \times_X (\PmPmo) &= \{ (K,L,x,x') \, | \, \T_x X , \T_{x'} X \subset L , \T_x X \subset K\} \\  
    &\subset X_{n+1}^* \times (X_{n+1}^*)^{\circ} \times X \times X\\
    &\subset \G(m,\P^N) \times \G(m,\P^N) \times X \times X,
  \end{align*}
  which is the fiber product of
  the projection $  \PmPmo \arw X : (L, x,x') \mapsto x $  and $ P^X_{n+1} \arw X$.

  For general $(K,L, x,x') \in    P^X_{n+1} \times_X (\PmPmo)$,
  $\langle K,L \rangle$ is an $(n+2)$-plane which contains $\T_x X, \T_{x'} X$.
  Hence we have rational maps
  \begin{align*}
    f  &:    P^X_{n+1} \times_X (\PmPmo) \dashrightarrow P_{n+2}^X : (K, L, x,x') \mapsto (\langle K,L \rangle, x) \\
    f' &:  P^X_{n+1} \times_X (\PmPmo) \dashrightarrow P_{n+2}^X : (K, L, x,x') \mapsto (\langle K,L \rangle, x').
  \end{align*}
  The rest is to see that $f,f'$ are dominant.

  For general $(M,x) \in P_{n+2}^X$,
  take general $(n+1)$-planes $K,L \subset M$ containing $\T_x X$.
  Since $(M,x) \in P_{n+2}^X$ is general, we have $(K,x), (L,x) \in (P_{n+1}^X)^{\circ}$.
  For $x' \in \sigma_{n+1}^L$,
  $(K,L,x,x')$ is an element in $P^X_{n+1} \times_X (\PmPmo)$ and $f(K,L,x,x') = (M,x)$.
  Hence $f$ is dominant.

  For general $(M,x') \in P_{n+2}^X$,
  take general $(n+1)$-planes $L \subset M$ which contains $\T_{x'} X$.
  Since $(M,x') \in P_{n+2}^X$ is general, we have $(L,x') \in (P_{n+1}^X)^{\circ}$.
  Take $x \in \sigma_{n+1}^L$.
  Then $(L,x) $ is in $ (P_{n+1}^X)^{\circ}$.
  For a general $(n+1)$-plane $K $ which contains $\T_x X$,
  $(K,L,x,x')$ is an element in $  P^X_{n+1} \times_X (\PmPmo)$ and $f' (K,L,x,x') = (M,x')$.
  Hence $f'$ is dominant.
\end{proof}

Lemmas \ref{thm:sig_m-fib} and \ref{thm:Lx=Lx'} imply
the following result.

\begin{prop}\label{thm:lin-Lam-x}
  Assume that $\gamma_{n+1}$ is separable
  and $\delta_{n+1} = 1$.
  Then the first projection
  $\Lambda \rightarrow X$ of \ref{eq:Lambda-in-XxX}
  induces a rational map
  \[
  X \dashrightarrow \Gr(N-n, \PN)
  \]
  sending $x \mapsto \Lmx$.
  Let $Y \subset \Gr(N-n,\PN)$ be the closure of the image of this rational map.
  For the universal family $\calu_Y \subset Y \times \PN$,
  the second projection $\calu_Y \rightarrow \PN$
  is birational onto $X \subset \PN$.
\end{prop}

\begin{proof}
  Since $\Lmx$ is a linear variety of dimension $N-n$ for general $x$, the rational map $X \dashrightarrow \Gr(N-n,\PN)$ is obtained.
  We consider the graph $g: X \dashrightarrow Y \times \PN$, which sends $x \mapsto (\Lmx, x)$. Since $x \in \Lmx$, the image of $g$ is contained in $\calu_Y$.

  In fact, $g$ is dominant to $\calu_Y$.
  This is because, for general $(\Lambda, x') \in \calu_Y$, we can write $\Lambda = \Lmx$ for some $x$, and then $\Lmx = \Lambda_{x'}$ by \autoref{thm:Lx=Lx'}, which implies that $g(x') = (\Lambda, x')$.

  Since the composite map $X \dashrightarrow \calu_Y \rightarrow X$ is identity, $\calu_Y \rightarrow X$ is birational.
\end{proof}

In fact, we have:

\begin{prop}\label{thm:U_Y-to-X-finite}
  In \autoref{thm:lin-Lam-x}, $\calu_Y \rightarrow X$ is an isomorphism.
\end{prop}

\begin{proof}
  \begin{inparaenum}[\itshape{}Step 1.]
  \item 
    First we show that $\Lambda \subset \pi(\gamma_{n+1}^{-1}(X_{n+1,x}^*))$ holds
    for  any $(\Lambda,x) \in \calu_Y$   as follows.
    Let us define a morphism
    \begin{equation}\label{eq:UYUY-to-GG}
      \calu_Y \times_Y \calu_Y \simeq \set*{(\Lambda, x_1,x_2) \in Y \times \PN \times \PN}{x_1,x_2 \in \Lambda}
      \rightarrow \Gn \times \Gn
    \end{equation}
    by $(\Lambda, x_1,x_2) \mapsto (\TT_{x_1}X, \TT_{x_2}X)$, and define $W$ to be the set of
    $(T_1,T_2) \in \Gn \times \Gn$ such that $\dim \lin{T_1, T_2} \leq n+1$. Note that $W$ is a closed subset of $\Gm \times \Gm$.

    In the above setting, the image of \ref{eq:UYUY-to-GG} is contained in $W$, as follows:
    A general member of the left hand side of \ref{eq:UYUY-to-GG}
    is written as $(\Lambda_x, x,x')$ with general points $x \in X$ and $x' \in \Lambda_x$.
    By definition, $x'$ is contained in $\sigma_{n+1}(L) = \pi(\gamma_{n+1}^{-1}(L))$
    with some $(n+1)$-plane $L \in (X_{n+1,x}^*)\spcirc$.
    This means that $\TT_{x'}X \subset L$.
    Since $\lin{\TT_{x}X, \TT_{x'}X} \subset L$, we have the assertion.

    Take $(\Lambda,x) \in \calu_Y$.   For any $x' \in \Lambda$, we have $\dim\lin{\TT_{x}X, \TT_{x'}X} \leq n+1$
    since \ref{eq:UYUY-to-GG} maps $(\Lambda, x,x')$ into $W$.
    Let $L$ be an $(n+1)$-plane containing $\lin{\TT_{x}X, \TT_{x'}X}$.
    Then $(L,x') \in P_{n+1}^X$ and $L \in X_{n+1,x}^*$, which means that
    $x' \in \pi(\gamma_{n+1}^{-1}(X_{n+1,x}^*))$.
    Thus $\Lambda \subset \pi(\gamma_{n+1}^{-1}(X_{n+1,x}^*))$ holds.
    \vspace{2mm}

  \item 
    Next we show the finiteness of $\calu_Y \rightarrow X$.
    Fix $x \in X$ to be any point.
    By a theorem of Zak \cite[I, 2.3 Theorem]{Zak}, 
    \emph{every} fiber of $\gamma_{n+1}:P_{n+1}^X \rightarrow X_{n+1}^*$ is of dimension $\leq 1$.
    Therefore every irreducible component of $\gamma_{n+1}^{-1}(X_{n+1,x}^*)$
    is of dimension $\leq (N-n-1)+1 = N-n$.
    From Step 1, every $(\Lambda,x) \in \calu_Y$ satisfies $\Lambda \subset \gamma_{n+1}^{-1}(X_{n+1,x}^*)$. Since $\dim(\Lambda) = N-n$, in fact $\Lambda$ coincides with an irreducible component of $\gamma_{n+1}^{-1}(X_{n+1,x}^*)$. Hence there only exist finitely many $\Lambda$'s such that $(\Lambda, x) \in \calu_Y$. It means that
    $\calu_Y \rightarrow X$ is finite.

    By Zariski's main theorem, we find that $\calu_Y \rightarrow X$ is isomorphic
    since it is birational by \autoref{thm:lin-Lam-x} and  $X$ is smooth.
  \end{inparaenum}
\end{proof}

By Zak's inequality,
we know $\delta_{n+i} \leq i$ for smooth $X \subset \PN$.
We study the case when the equality holds.
The following theorem implies \autoref{thm:gamma_n+1-sep-bir-unlessP1Pn-1}

\begin{thm}\label{thm:delta_m=m-n}
  Let $X \subset \PN$ be
  a non-degenerate smooth projective variety
  with $n := \dim(X) < N-2$.
  Assume that $\gamma_{n+1}$ is separable.
  Then the following conditions are equivalent:
  \begin{enumerate}[\quad \normalfont{}(a)]
  \item \label{item:1}
    $\delta_{n+i} = i$ holds for an integer $i > 0$ with $n+i < N-1$.
  \item 
    $\delta_{n+i} = i$ holds for any integer $i \geq 0$ with $n+i \leq N-1$.
  \item $X$ is the image of the Segre embedding $\P^1 \times \P^{n-1} \hookrightarrow \P^{2n-1}$.
  \end{enumerate}
\end{thm}

\begin{rem}\label{thm:delta_m=m-n_rem}
  Assume that the characteristic is zero. Then
  Ein \cite[Theorem~3.4]{Ein} showed that
  if $\dim(X) = N-2$ and $\delta_{N-1} (= \delta_{n+1}) = 1$,
  then $X = \PP^1 \times \PP^2 \subset \PP^5$.
  Hence combining with \autoref{thm:delta_m=m-n}, we have that
  $\PP^1 \times \PP^{n-1} \subset \PP^{2n-1}$
  is only the smooth projective variety of codimension $\geq 2$
  which satisfies $\delta_{n+i} = i$ for all $i$ with $0 \leq i \leq N-n-1$.

\end{rem}

\begin{proof}[Proof of \autoref{thm:delta_m=m-n}]
  (c) $\Rightarrow$ (b) follows from \autoref{thm:P1Pn-1}.
  (b) $ \Rightarrow $ (a) follows immediately.
  We show (a) $ \Rightarrow $ (c) as follows.

  From \autoref{thm:eq-Zaks-formula}, we have $\delta_{n+1} = 1$.
  From \autoref{thm:U_Y-to-X-finite},
  $X$ is isomorphic to the projective bundle $\calu_Y$ with $\dim(Y) = 2n-N$.
  In particular, $\Pic(X)$ is of rank $\geq 2$.

  If $n \geq (N+2)/2$, then $\Pic(X) = \ZZ$ if the characteristic $p=0$ due to the Barth-Larsen theorem (see \cite[Corollary 3.2.3]{Laz} for example)
  and $\Pic(X)$ is a finitely generated abelian group of rank $1$
  if $p > 0$ due to \cite[Theorem~(3.1)]{Speiser}, respectively.
  Therefore $n < (N+2)/2$ must hold, which implies $\dim(Y) = 1$ and $N=2n-1$.

  Then we find that the isomorphism
  $\calu_Y \rightarrow X \subset \PP^{2n-1}$
  is in fact the Segre embedding;
  see 
  \cite[pp. 307--308]{Kleiman1982} or
  \cite[Proposition 3.1]{SIERRA}.
\end{proof}

\begin{proof}[Proof of \autoref{thm:gamma_n+1-sep-bir-unlessP1Pn-1}]
  Assume that $\gamma_{n+1}$ is separable.
  If $X$ is not the Segre embedding, then it follows from
  \autoref{thm:delta_m=m-n} that $\delta_{n+1} = 0$ must hold.
  Since the contact locus $\pi(\gamma_{n+1}^{-1}(L)) \subset X$ of general $L \in X_{n+1}^*$ is a linear variety because of \autoref{thm:m-th-linearity}, a general fiber is a point.
  This means that $\gamma_{n+1}$ is birational.
\end{proof}

\end{document}